\documentclass[12pt,reqno]{amsart}
\usepackage{amsmath,amssymb,amsfonts,amscd,latexsym,amsthm,mathrsfs,verbatim,comment,cite}
\usepackage{hyperref}
\newtheorem{theorem}{Theorem}[section]

\newtheorem{lemma}[theorem]{Lemma}
\newtheorem{corollary}[theorem]{Corollary}

\renewcommand{\pmod}[1]{\;\allowbreak(\operatorname{mod}#1)}

\numberwithin{equation}{section}

\begin{document}

\title{A common $q$-analogue of two supercongruences}

\date{24 October 2019}

\author{Victor J. W. Guo}
\address{School of Mathematics and Statistics, Huaiyin Normal University, Huai'an 223300, Jiangsu, People's Republic of China}
\email{jwguo@hytc.edu.cn}

\author{Wadim Zudilin}
\address{Department of Mathematics, IMAPP, Radboud University, PO Box 9010, 6500~GL Nijmegen, Netherlands}
\email{w.zudilin@math.ru.nl}

\thanks{The first author was supported by the National Natural Science Foundation of China (grant 11771175).
The second author was supported by JSPS Invitational Fellowships for Research in Japan (fellowship S19126).}


\begin{abstract}
We give a $q$-congruence whose specialisations $q=-1$ and $q=1$ correspond to supercongruences (B.2) and (H.2) on Van Hamme's 1997 list:
$$
\sum_{k=0}^{(p-1)/2}(-1)^k(4k+1)A_k\equiv p(-1)^{(p-1)/2}\pmod{p^3}
\quad\text{and}\quad
\sum_{k=0}^{(p-1)/2}A_k\equiv a(p)\pmod{p^2},
$$
where $p>2$ is prime,
$$
A_k=\prod_{j=0}^{k-1}\biggl(\frac{1/2+j}{1+j}\biggr)^3=\frac1{2^{6k}}{\binom{2k}k}^3
\quad\text{for}\; k=0,1,2,\dots,
$$
and $a(p)$ is the $p$-th coefficient of (the weight 3 modular form) $q\prod_{j=1}^\infty(1-q^{4j})^6$.
We complement our result with a general common $q$-congruence for related hypergeometric sums.
\end{abstract}

\maketitle

\section{Introduction}
\label{sec1}

The formula of Bauer \cite{Bauer} from 1859,
\begin{equation}
\sum_{k=0}^\infty(-1)^k(4k+1)A_k=\frac2\pi,
\quad\text{where}\; A_k=\frac1{2^{6k}}{\binom{2k}k}^3 \;\;\text{for}\; k=0,1,2,\dots,
\label{eq:ram}
\end{equation}
is one of traditional targets for different methods of proofs of hypergeometric identities.
Its special status is probably linked to the fact that it belongs to a family of series for $1/\pi$ of Ramanujan type,
after Ramanujan \cite{Ramanujan} brought to life in 1914 a long list of similar looking equalities for the constant but with a faster convergence.
Identity \eqref{eq:ram} is a particular instance of $_5F_4$ hypergeometric summation (known to Ramanujan)
but there are several proofs of it, including the original one \cite{Bauer} of Bauer, that do not require any knowledge of hypergeometric functions.
One notable\,---\,computer\,---\,proof of \eqref{eq:ram} was given in 1994 by Ekhad and Zeilberger \cite{EZ}
using the Wilf--Zeilberger (WZ) method of creative telescoping.

It was observed in 1997 by Van Hamme \cite{Hamme} that many Ramanujan's and Ramanu\-jan-like evaluations have nice $p$-adic analogues;
for example, the congruence
\begin{equation}
\sum_{k=0}^{(p-1)/2}(-1)^k(4k+1)A_k\equiv p(-1)^{(p-1)/2}\pmod{p^3}
\label{eq:pram}
\end{equation}
(tagged (B.2) on Van Hamme's list) is valid for any prime $p>2$ and corresponds to the equality \eqref{eq:ram}.
The congruence \eqref{eq:pram} was first proved by Mortenson \cite{Mortenson} using a $_6F_5$ hypergeometric transformation; it later received another proof by one of these authors \cite{Zudilin} via the WZ method
(in fact, using the very same `WZ certificate' as in~\cite{EZ} for~\eqref{eq:ram}).
Notice that \eqref{eq:pram} is an example of \emph{supercongruence} meaning that it holds modulo a power of~$p$ greater than~1.

Another entry on Van Hamme's 1997 list \cite{Hamme}, tagged (H.2), is the congruence
\begin{equation}
\sum_{k=0}^{(p-1)/2}A_k
\equiv\begin{cases}
-\Gamma_p(1/4)^4  \pmod{p^2} &\text{if $p\equiv 1\pmod 4$}, \\
0\pmod{p^2} &\text{if $p\equiv 3\pmod 4$},
\end{cases}
\label{eq:h2}
\end{equation}
again for any $p>2$ prime, and $\Gamma_p(x)$ is the $p$-adic Gamma function.
Van Hamme not only observed but also proved \eqref{eq:h2} in \cite{Hamme},
and it was later generalized by Z.-H.~Sun \cite[Theorem 2.5]{Sun2}, Guo and Zeng~\cite[Corollary~1.2]{GZ15},
Long and Ramakrishna \cite{LR}, Liu \cite{Liu0}, \cite[Theorem 1.5]{Liu} in different ways.
For example, Long and Ramakrishna \cite[Theorem 3]{LR} gave the following generalization of~\eqref{eq:h2}:
\begin{equation}
\sum_{k=0}^{(p-1)/2} A_k
\equiv
\begin{cases}
-\Gamma_p(1/4)^4  \pmod{p^3} &\text{if $p\equiv 1\pmod 4$},\\[2.5pt]
-\dfrac{p^2}{16}\,\Gamma_p(1/4)^4\pmod{p^3} &\text{if $p\equiv 3\pmod 4$}.
\end{cases}
\label{eq:lr}
\end{equation}
Recently, these authors \cite[Theorem~2]{GuoZu2} proved that, for any positive odd integer $n$, modulo $\Phi_n(q)^2$,
\begin{align}
\sum_{k=0}^{(n-1)/2}\frac{(q;q^2)_k^2(q^2;q^4)_k}{(q^2;q^2)_k^2(q^4;q^4)_k}\,q^{2k}
\equiv\begin{cases}
\dfrac{(q^2;q^4)_{(n-1)/4}^2}{(q^4;q^4)_{(n-1)/4}^2}\,q^{(n-1)/2} &\text{if}\; n\equiv1\pmod4, \\[2.5pt]
0 &\text{if}\; n\equiv3\pmod4.
\end{cases}
\label{eq:mod-phi}
\end{align}
Here and in what follows, $\Phi_n(q)$ denotes the $n$-th {\em cyclotomic polynomial};
the {\em $q$-shifted factorial} is given by $(a;q)_0=1$ and $(a;q)_n=(1-a)(1-aq)\dotsb(1-aq^{n-1})$ for $n\geqslant 1$ or $n=\infty$,
while $[n]=[n]_q=1+q+\dots+q^{n-1}$ stands for the $q$-{\em integer}.
Van Hamme \cite[Theorem 3]{Hamme0} also proved that
$$
\binom{-1/2}{(p-1)/4}\equiv-\frac{\Gamma_p(1/4)^2}{\Gamma_p(1/2)}\pmod{p^2};
$$
in view of $\Gamma_p(1/2)^2=-1$ for $p\equiv 1\pmod{4}$, by letting $q\to 1$ in \eqref{eq:mod-phi} for $n=p$
we immediately obtain \eqref{eq:h2}.

One feature of \eqref{eq:h2} (not highlighted in \cite{Hamme}) is its connection with the coefficients
\begin{equation}
a(p)=\begin{cases}
2(a^2-b^2) & \text{if $p=a^2+b^2$, $a$ odd}, \\
0 & \text{if $p\equiv3\pmod4$},
\end{cases}
\label{eq:a(p)}
\end{equation}
of CM modular form $q\prod_{j=1}^\infty(1-q^{4j})^6$ of weight~3, namely, the congruence
$$
a(p)\equiv-\Gamma_p(1/4)^4\pmod{p^2}
\quad\text{for primes}\; p\equiv 1\pmod 4.
$$
This served as a main motivation in \cite{GuoZu2} for not only establishing \eqref{eq:mod-phi} but also speculating on possible $q$-deformation of modular forms.

For some other recent progress on $q$-analogues of supercongruences, the reader is referred to
\cite{Gorodetsky,Guo2018,Guo-m3,Guo4.5,GS0,GS,GuoZu,NP,Straub,Tauraso2,Zudilin}.
In particular, the authors \cite{GuoZu} introduced and executed a new method of creative microscoping to prove (and reprove) many $q$-analogues of classical supercongruences and also raised some problems on $q$-cong\-ruences.
Using this method, the first author \cite{Guo-new} gave a refinement of \eqref{eq:mod-phi} modulo $\Phi_n(q)^3$
for $n\equiv 3\pmod{4}$, in other words, a $q$-analogue of \eqref{eq:lr} for $p\equiv 3\pmod{4}$.

A goal of this note is to present the following new $q$-analogue of Van Hamme's supercongruence \eqref{eq:h2}.

\begin{theorem}
\label{thm:main-1}
Let $n$ be a positive odd integer. Then
\begin{align}
&\sum_{k=0}^{(n-1)/2}\frac{(1+q^{4k+1})\,(q^2;q^4)_k^3}{(1+q)\,(q^4;q^4)_k^3}\,q^{k} \notag\\
&\quad
\equiv\dfrac{[n]_{q^2}(q^3;q^4)_{(n-1)/2}}{(q^5;q^4)_{(n-1)/2}}\,q^{(1-n)/2}
\begin{cases}
\kern-4.5pt \pmod{\Phi_n(q)^2 \Phi_n(-q)^3} &\text{if}\; n\equiv1\pmod4, \\[9pt]
\kern-4.5pt \pmod{\Phi_n(q)^3 \Phi_n(-q)^3} &\text{if}\; n\equiv3\pmod4.
\end{cases}
\label{eq:mod-phi-cases}
\end{align}
\end{theorem}

Note that $\Phi_n(q)\Phi_n(-q)=\Phi_n(q^2)$ for odd indices $n$.

The $n\equiv 3\pmod{4}$ case of Theorem~\ref{thm:main-1}
confirms a conjecture of these authors \cite[Conjecture 4.13]{GuoZu},
which states that, for $n\equiv 3\pmod{4}$,
$$
\sum_{k=0}^{(n-1)/2}\frac{(1+q^{4k+1})\,(q^2;q^4)_k^3}{(1+q)\,(q^4;q^4)_k^3}\,q^{k}
\equiv 0\pmod{\Phi_n(q)^2 \Phi_n(-q)}.
$$

It is not difficult to verify that
$$
\frac{(3/4)_{(p-1)/2}}{(5/4)_{(p-1)/2}}
\equiv -\frac{p}{16}\Gamma_p\left(1/4\right)^4 \pmod{p^2}
$$
for $p\equiv 3\pmod{4}$, where
$(a)_n=a(a+1)\cdots(a+n-1)$ denotes the rising factorial (also known as Pochhammer's symbol).
Therefore, the $q$-congruence \eqref{eq:mod-phi-cases} reduces to \eqref{eq:lr}
for $p\equiv 3\pmod{4}$ when $n=p$ and $q\to 1$, and it reduces to \eqref{eq:h2}
for $p\equiv 1\pmod{4}$ when $n=p$ and $q\to 1$.
Moreover, letting $n=p$ and $q\to-1$ in \eqref{eq:mod-phi-cases}, we immediately get \eqref{eq:pram}.
Thus, Theorem~\ref{thm:main-1} presents a common $q$-analogue of supercongruences \eqref{eq:pram} and \eqref{eq:h2}.
We point out that other different $q$-analogues of \eqref{eq:pram} have been given in \cite{Guo2018,Guo-m3}.

Recently, Mao and Pan \cite{MP} (see also Sun \cite[Theorem 1.3(i)]{Sun3}) proved that, if $p\equiv 1\pmod{4}$ is a prime, then
\begin{align}
\sum_{k=0}^{(p+1)/2}\frac{(-1/2)_k^3}{k!^3}\equiv 0\pmod{p^2}.
\label{eq:neg-1-p}
\end{align}
In this note, we prove the following $q$-analogue of \eqref{eq:neg-1-p}.

\begin{theorem}\label{thm:main-2}
Let $n>1$ be an odd integer. Then
\begin{align}
&\sum_{k=0}^{(n+1)/2}\frac{(1+q^{4k-1})\,(q^{-2};q^4)_k^3}{(1+q)\,(q^4;q^4)_k^3}\,q^{7k} \notag\\
&\quad
\equiv\dfrac{[n]_{q^2}(q;q^4)_{(n-1)/2}}{(q^7;q^4)_{(n-1)/2}}\,q^{(n-3)/2}
\begin{cases}
\kern-4.5pt \pmod{\Phi_n(q)^3 \Phi_n(-q)^3} &\text{if}\; n\equiv1\pmod4, \\[9pt]
\kern-4.5pt \pmod{\Phi_n(q)^2 \Phi_n(-q)^3} &\text{if}\; n\equiv3\pmod4.
\end{cases}
\notag
\end{align}
\end{theorem}

For $n$ prime, letting $q\to 1$ in Theorem \ref{thm:main-2} we obtain the following generalization of~\eqref{eq:neg-1-p}.

\begin{corollary}
\label{cor1.3}
Let $p$ be an odd prime. Then
\begin{equation*}
\sum_{k=0}^{(p+1)/2} \frac{(-1/2)_k^3}{k!^3}
\equiv
p\,\dfrac{(1/4)_{(p-1)/2}}{(7/4)_{(p-1)/2}}
\begin{cases}
\kern-4.5pt \pmod{p^3} &\text{if $p\equiv 1\pmod 4$,}\\[9pt]
\kern-4.5pt \pmod{p^2} &\text{if $p\equiv 3\pmod 4$.}
\end{cases}
\end{equation*}
\end{corollary}

On the other hand, for $n$ prime and $q\to -1$ in Theorem~\ref{thm:main-2},
we are led to the following result:
\begin{equation}
\sum_{k=0}^{(p+1)/2} (-1)^k (4k-1)\frac{(-1/2)_k^3}{k!^3}
\equiv p(-1)^{(p+1)/2} \pmod{p^3}.
\label{eq:side-res}
\end{equation}
It should be mentioned that a different $q$-analogue of \eqref{eq:side-res} was given in \cite[Theorem~4.9]{GuoZu} with $r=-1$, $d=2$ and $a=1$ (see also \cite[Section~5]{GS}).

Both Theorems \ref{thm:main-1} and \ref{thm:main-2} are particular cases of a more general result, which we state and prove in the next section.

\section{A family of $q$-congruences from the $q$-Dixon sum}
\label{sec2}

In this section we establish the following one-parameter family of $q$-congruences.

\begin{theorem}\label{thm:main-3}
\label{thm:main}
Let $n\geqslant1$ be an odd integer
and $\ell$ an integer with $0\leqslant \ell\leqslant (n-1)/2$.
Then
\begin{align}
&\sum_{k=0}^{n-1}\frac{(1+q^{4k-2\ell+1})\,(q^{2-4\ell};q^4)_k^3}{(1+q^{1-2\ell })\,(q^4;q^4)_k^3}\,q^{(6\ell+1)k} \notag \\
&\equiv \frac{(1-q^{2n})\,(q^{3-6\ell};q^4)_{(n-1)/2+\ell}}
{(1-q^{2-4\ell})\,(q^{5-2\ell};q^4)_{(n-1)/2+\ell}}\,q^{(2\ell-1)((n-1)/2+\ell)}
\begin{cases}
\kern-4.5pt \pmod{\Phi_n(q)^2 \Phi_n(-q)^3} \\ \qquad \text{if}\; n+2\ell\equiv1\pmod4, \\[3pt]
\kern-4.5pt \pmod{\Phi_n(q)^3 \Phi_n(-q)^3} \\ \qquad \text{if}\; n+2\ell\equiv3\pmod4.
\end{cases}
\label{eq:mod-ell}
\end{align}
\end{theorem}

Note that the $q$-congruence \eqref{eq:mod-ell} remains true when the sum is over $k$ from 0 to $(n-1)/2+\ell$,
since $(q^{2-4\ell};q^4)_k/(q^4;q^4)_k\equiv 0\pmod{\Phi_n(q^2)}$ for $(n-1)/2+\ell<k\leqslant n-1$.
Furthermore, when $\ell=0$ and $\ell=1$ (hence $n\ge3$) the theorem reduces to Theorems \ref{thm:main-1} and \ref{thm:main-2}, respectively.

The following easily proved $q$-congruence (see \cite[Lemma 3.1]{GS}) is necessary in our derivation of Theorem~\ref{thm:main-3}.

\begin{lemma}\label{lem:2.1}
Let $n$ be a positive odd integer.
Then, for $0\leqslant k\leqslant (n-1)/2$, we have
\begin{equation*}
\frac{(aq;q^2)_{(n-1)/2-k}}{(q^2/a;q^2)_{(n-1)/2-k}}
\equiv (-a)^{(n-1)/2-2k}\frac{(aq;q^2)_k}{(q^2/a;q^2)_k}\,q^{(n-1)^2/4+k}
\pmod{\Phi_n(q)}.
\end{equation*}
\end{lemma}

Like the proofs given in \cite{GuoZu}, we start with the following generalization of \eqref{eq:mod-phi-cases} with an extra parameter~$a$.

\begin{theorem}
Let $n>1$ be an odd integer and $0\leqslant\ell\leqslant (n-1)/2$. Then
\begin{align}
&\sum_{k=0}^{n-1}\frac{(1+q^{4k-2\ell+1})\,(aq^{2-4\ell};q^4)_k (q^{2-4\ell}/a;q^4)_k (q^{2-4\ell};q^4)_k}
{(1+q^{1-2\ell })\,(aq^4;q^4)_k (q^4/a;q^4)_k (q^4;q^4)_k}\,q^{(6\ell+1)k} \notag \\
&\equiv
\frac{(1-q^{2n})\,(q^{3-6\ell};q^4)_{(n-1)/2+\ell}}
{(1-q^{2-4\ell})\,(q^{5-2\ell};q^4)_{(n-1)/2+\ell}}\,q^{(2\ell-1)((n-1)/2+\ell)}
\begin{cases}
\kern-4.5pt \pmod{\Phi_n(-q) (1-aq^{2n})(a-q^{2n})} \\ \qquad\text{if}\; n+2\ell\equiv1\pmod4, \\[3pt]
\kern-4.5pt \pmod{\Phi_n(q^2) (1-aq^{2n})(a-q^{2n})} \\ \qquad\text{if}\; n+2\ell\equiv3\pmod4.
\end{cases}
\label{eq:mod-ell-a}
\end{align}
\end{theorem}

\begin{proof}
Performing the parameter substitutions $q\mapsto q^4$, $a\mapsto q^{2-4\ell}$, $b\mapsto bq^{2-4\ell}$ and $c\mapsto cq^{2-4\ell}$ in the $q$-Dixon sum \cite[Appendix (II.13)]{GR}, we obtain
\begin{align}
&
\sum_{k=0}^\infty \frac{(1+q^{4k-2\ell+1})\,(q^{2-4\ell};q^4)_k (bq^{2-4\ell};q^4)_k (cq^{2-4\ell};q^4)_k}
{(1+q^{1-2\ell})\,(q^4/b;q^4)_k (q^4/c;q^4)_k (q^4;q^4)_k} \biggl(\frac{q^{6\ell+1}}{bc}\biggr)^k
\nonumber\\ &\qquad
=\frac{(q^{6-4\ell};q^4)_\infty (q^{2\ell+3}/b;q^4)_\infty (q^{2\ell+3}/c;q^4)_\infty (q^{4\ell+2}/bc;q^4)_\infty}
{(q^4/b;q^4)_\infty (q^4/c;q^4)_\infty (q^{5-2\ell};q^4)_\infty (q^{6\ell+1}/bc;q^4)_\infty}.
\label{eq:qDixon-3}
\end{align}
Since $n$ is odd, putting $b=q^{-2n}$ and $c=q^{2n}$ in \eqref{eq:qDixon-3} we see that the left-hand side terminates and is equal to
\begin{align*}
&
\sum_{k=0}^{(n-1)/2+\ell}\frac{(1+q^{4k-2\ell+1})\,(aq^{2-4\ell};q^4)_k (q^{2-4\ell}/a;q^4)_k (q^{2-4\ell};q^4)_k}
{(1+q^{1-2\ell })\,(aq^4;q^4)_k (q^4/a;q^4)_k (q^4;q^4)_k}\,q^{(6\ell+1)k}
\\ &\quad
=\sum_{k=0}^{n-1}\frac{(1+q^{4k-2\ell+1})\,(aq^{2-4\ell};q^4)_k (q^{2-4\ell}/a;q^4)_k (q^{2-4\ell};q^4)_k}
{(1+q^{1-2\ell })\,(aq^4;q^4)_k (q^4/a;q^4)_k (q^4;q^4)_k}\,q^{(6\ell+1)k},
\end{align*}
while the right-hand side becomes
\begin{align*}
&
\frac{(q^{2\ell-2n+3};q^4)_{(n-1)/2+\ell} (q^{6-4\ell};q^4)_{(n-1)/2+\ell}}
{(q^{4-2n};q^4)_{(n-1)/2+\ell} (q^{5-2\ell};q^4)_{(n-1)/2+\ell}  }
\\ &\quad
=\frac{(1-q^{2n})\,(q^{3-6\ell};q^4)_{(n-1)/2+\ell}}
{(1-q^{2-4\ell})\,(q^{5-2\ell};q^4)_{(n-1)/2+\ell}}\,q^{(2\ell-1)((n-1)/2+\ell)}.
\end{align*}
This proves that the $q$-congruence \eqref{eq:mod-ell-a} holds modulo $1-aq^{2n}$ or $a-q^{2n}$.

On the other hand, by Lemma \ref{lem:2.1}, for $0\leqslant k\leqslant (n-1)/2+\ell$,
modulo $\Phi_n(q)$ we have
\begin{align*}
\frac{(aq^{1-2\ell};q^2)_{(n-1)/2+\ell-k}}{(q^2/a;q^2)_{(n-1)/2+\ell-k}}
&=\frac{(aq^{1-2\ell};q^2)_{\ell}(aq;q^2)_{(n-1)/2-k}}
{(q^{n+1-2k}/a;q^2)_{\ell}(q^2/a;q^2)_{(n-1)/2-k}}
\\
&\equiv (-a)^{(n-1)/2-2k}\frac{(aq^{1-2\ell};q^2)_{\ell}(aq;q^2)_k}
{(q^{n+1-2k}/a;q^2)_{\ell}(q^2/a;q^2)_k}\,q^{(n-1)^2/4+k}
\\
&=(-a)^{(n-1)/2-2k}\frac{(aq^{1-2\ell};q^2)_k (aq^{2k-2\ell+1};q^2)_\ell }{(q^{n+1-2k}/a;q^2)_{\ell}(q^2/a;q^2)_k}\,q^{(n-1)^2/4+k}
\\
&\equiv(-a)^{(n-1)/2+\ell-2k}\frac{(aq^{1-2\ell};q^2)_k }{(q^2/a;q^2)_k}\,q^{(n-1)^2/4+k+(2k-\ell)\ell},
\end{align*}
where we used $q^n\equiv 1\pmod{\Phi_n(q)}$ in the last step.
Using the above $q$-congruence we can easily check that, for odd $n>1$ and $0\leqslant k\leqslant (n-1)/2+\ell$,
sum of the $k$-th  and $((n-1)/2+\ell-k)$-th summands on the left-hand side of \eqref{eq:mod-ell-a} is congruent to $0$ modulo $\Phi_n(-q)$  (or modulo $\Phi_n(q^2)$ if $n\equiv 3-2\ell\pmod{4}$). It follows that
\begin{align*}
&\sum_{k=0}^{(n-1)/2+\ell}
\frac{(1+q^{4k-2\ell+1})\,(aq^{2-4\ell};q^4)_k (q^{2-4\ell}/a;q^4)_k (q^{2-4\ell};q^4)_k}
{(1+q^{1-2\ell })\,(aq^4;q^4)_k (q^4/a;q^4)_k (q^4;q^4)_k}\,q^{(6\ell+1)k}
\\
&\quad\equiv 0
\begin{cases}
\kern-4.5pt \pmod{\Phi_n(-q)} &\text{if $n+2\ell\equiv 1\pmod{4}$,} \\[4.5pt]
\kern-4.5pt \pmod{\Phi_n(q^2)} &\text{if $n+2\ell \equiv 3\pmod{4}$.}
\end{cases}
\end{align*}
It is easy to see that the right-hand side of \eqref{eq:mod-ell} is congruent to $0$ modulo  $\Phi_n(-q)$ if $n+2\ell\equiv 1\pmod{4}$
and modulo $\Phi_n(q^2)$ if $n+2\ell\equiv 3\pmod{4}$.
Therefore, the $q$-congruence \eqref{eq:mod-ell-a} holds modulo $\Phi_n(-q)$ if $n+2\ell\equiv 1\pmod{4}$
and modulo $\Phi_n(q^2)$ if $n+2\ell\equiv 3\pmod{4}$.
Since the polynomials $1-aq^{2n}$, $a-q^{2n}$ and $\Phi_n(-q)$ (or $\Phi_n(q^2)$) are pairwise coprime, we complete the proof of~\eqref{eq:mod-ell-a}.
\end{proof}

\begin{proof}[Proof of Theorem \textup{\ref{thm:main-3}}]
We assume that $n>1$, since the $n=1$ case (making $\ell=0$ only possible) is trivial.
The limits of the denominators on both sides of \eqref{eq:mod-ell-a} as $a\to 1$
are relatively prime to $\Phi_n(q^2)$, since $k$ is in the range $0\leqslant k\leqslant (n-1)/2+\ell$.
On the other hand, the limit of $(1-aq^{2n})(a-q^{2n})$ as $a\to1$ contains the factor $\Phi_n(q^2)^2$.
\end{proof}

\section{Discussion}
\label{sec3}

The method of creative microscoping used in our proofs indicates the origin of $q$-congruences from \emph{infinite} $q$-hypergeometric identities; for example, the $q$-congruence \eqref{eq:mod-phi-cases} corresponds to the identity
\begin{equation}
\sum_{k=0}^\infty\frac{(1+q^{4k+1})\,(q^2;q^4)_k^3}{(1+q)\,(q^4;q^4)_k^3}\,q^{k}
=\frac{(q^2;q^4)_\infty^2(q^3;q^4)_\infty^2}{(1+q)\,(q;q^4)_\infty^2(q^4;q^4)_\infty^2},
\label{eq:origin}
\end{equation}
which is just a particular instance of \eqref{eq:qDixon-3}. Note that the limiting cases as $q\to-1$ and $q\to1$ of \eqref{eq:origin} give the formulas \eqref{eq:ram} and
\begin{equation}
\sum_{k=0}^\infty\frac{(\frac12)_k^3}{k!^3}
=\frac{\Gamma(1/4)^4}{4\pi^3}
=\frac{8L(f,1)}{\pi}
=\frac{16L(f,2)}{\pi^2}
\label{eq01}
\end{equation}
where
\begin{equation*}
f(\tau)=q\prod_{j=1}^\infty(1-q^{4j})^6=\sum_{n=1}^\infty a(n)q^n,
\quad\text{with}\; q=\exp(2\pi i\tau),
\end{equation*}
is the CM modular form from the introduction and $L(f,s)$ denotes its $L$-function.
This means that the $q$-identity \eqref{eq:origin} presents a common $q$-extension of evaluations \eqref{eq:ram} and~\eqref{eq01}\,---\,the fact that makes it less surprising that the $q$-congruence \eqref{eq:mod-phi-cases} simultaneously extends \eqref{eq:pram} and \eqref{eq:h2}.

The intermediate use of \emph{parametric} $q$-hypergeometric identities in our proof of Theorem~\ref{thm:main-3} based on the $q$-Dixon sum suggests that different $q$-congruences underlying \eqref{eq:origin} are possible. This is indeed the case when we analyze the formula~\eqref{eq:origin} as the $a=1$ specialization of
\begin{align}
&
\sum_{k=0}^\infty\frac{(1+q^{4k+1})\,(aq;q^2)_k(q/a;q^2)_k(-q;q^2)_k^2(q^2;q^4)_k}
{(1+q)\,(q^2;q^2)_k^2(-aq^2;q^2)_k(-q^2/a;q^2)_k(q^4;q^4)_k}\,q^k
\nonumber\\ &\quad
=\frac{(-q;q^2)_\infty^2(aq^3;q^4)_\infty^2(q^3/a;q^4)_\infty^2}
{(1+q)\,(-aq^2;q^2)_\infty(-q^2/a;q^2)_\infty(q^2;q^2)_\infty^2}
\label{eq03}
\end{align}
which originates from a $q$-analogue of Watson's $_3F_2$ sum \cite[Appendix (II.16)]{GR}.
When we choose $a=q^n$ (or $a=q^{-n}$) in \eqref{eq03}, for $n>1$ odd, we get the sum terminating after $(n-1)/2$ terms on the left-hand side of \eqref{eq03}, while the right-hand side vanishes if $n$ is of the form $4m+3$ and it becomes equal to
\begin{align*}
\frac{(-q;q^2)_\infty^2(q^{4m+4};q^4)_\infty^2(q^{2-4m};q^4)_\infty^2}
{(1+q)\,(-q^{4m+3};q^2)_\infty(-q^{1-4m};q^2)_\infty(q^2,q^4;q^4)_\infty^2}
=[4m+1]\,\frac{(q^2;q^4)_m^2}{(q^4;q^4)_m^2}
\end{align*}
if $n=4m+1$. This means that modulo $(a-q^n)(1-aq^n)$ we have
\begin{align*}
&
\sum_{k=0}^N\frac{(1+q^{4k+1})\,(aq;q^2)_k(q/a;q^2)_k(-q;q^2)_k^2(q^2;q^4)_k}
{(1+q)\,(q^2;q^2)_k^2(-aq^2;q^2)_k(-q^2/a;q^2)_k(q^4;q^4)_k}\,q^k
\nonumber\\ &\quad
\equiv\begin{cases}
[4m+1]\,\dfrac{(q^2;q^4)_m^2}{(q^4;q^4)_m^2} &\text{if}\; n=4m+1, \\[2.5pt]
0 &\text{if}\; n\equiv3\pmod4,
\end{cases}
\end{align*}
for any $N\ge(n-1)/2$. The limiting $a\to1$ case of the congruences can be shown to be
\begin{align}
\sum_{k=0}^{(n-1)/2}\frac{(1+q^{4k+1})\,(q^2;q^4)_k^3}{(1+q)\,(q^4;q^4)_k^3}\,q^k
\equiv\begin{cases}
[4m+1]\,\dfrac{(q^2;q^4)_m^2}{(q^4;q^4)_m^2} &\text{if}\; n=4m+1, \\[2.5pt]
0 &\text{if}\; n\equiv3\pmod4,
\end{cases}
\label{eq05}
\end{align}
modulo $\Phi_n(q)^2\Phi_n(-q)$. This is quite similar in spirit to \eqref{eq:mod-phi}, though still far from constructing $q$-analogues for the coefficients $a(p)$ in~\eqref{eq:a(p)} of the modular form $f(\tau)$.
The latter means that a hunt for $q$-rational functions, which equal the left-hand side of \eqref{eq:mod-phi} or \eqref{eq05} modulo $\Phi_n(q)^2$ and specialize to $a(n)$ as $q\to1$ (at least for $n$~prime), is still on its way.
Such $q$-rational functions are also expected to be self-reciprocal, that is, invariant under the involution $q\mapsto1/q$,
as all the left- and right-hand sides in \eqref{eq:mod-phi}, \eqref{eq:mod-phi-cases}, \eqref{eq05} and also \eqref{eq:mod-ell} are.

\end{document}